\documentclass[12pt, a4paper]{amsart}
\pdfoutput=1
\usepackage{amsmath,amsfonts,amssymb,amsthm}  
\usepackage[utf8]{inputenc}
\usepackage{amsrefs}
\usepackage{tikz}
\usetikzlibrary{decorations.pathreplacing}

\theoremstyle{plain}
\newtheorem{theorem}{Theorem}[section]
\newtheorem{corollary}[theorem]{Corollary}
\newtheorem{lemma}[theorem]{Lemma}

\theoremstyle{definition}
\newtheorem{definition}[theorem]{Definition}

\newtheorem{remark}[theorem]{Remark}

\numberwithin{equation}{section}
\newtheorem*{theorem*}{Theorem}

\def\Xint#1{\mathchoice
{\XXint\displaystyle\textstyle{#1}}
{\XXint\textstyle\scriptstyle{#1}}
{\XXint\scriptstyle\scriptscriptstyle{#1}}
{\XXint\scriptscriptstyle\scriptscriptstyle{#1}}
\!\int}
\def\XXint#1#2#3{{\setbox0=\hbox{$#1{#2#3}{\int}$}
\vcenter{\hbox{$#2#3$}}\kern-.5\wd0}}

\def\dashint{\Xint-}

\newcommand{\dmu}{\, \mathrm{d}\mu}
\newcommand{\dx}{\, \mathrm{d}x}
\newcommand{\dt}{\, \mathrm{d}t}
\newcommand{\dla}{\, \mathrm{d} \lambda}

\DeclareMathOperator{\dive}{div}

\DeclareMathOperator{\BMO}{BMO}
\DeclareMathOperator{\PBMO}{PBMO}

\providecommand{\abs}[1]{ \lvert#1  \rvert}
\providecommand{\norm}[1]{ \lVert#1  \rVert}

\providecommand{\prt} {{ \rm pr }_{t} \, } 

\allowdisplaybreaks
 
\begin{document}

\title[Parabolic $\BMO$]{Parabolic $\BMO$ and global integrability of supersolutions to doubly nonlinear parabolic equations}

\author{Olli Saari} 

\address{(O.S.) Aalto University, Department of
  Mathematics and Systems Analysis, P.O. Box 11100, FI-00076 Aalto,
  Finland}\email{olli.saari@aalto.fi}
  


\begin{abstract} We prove that local and global parabolic $\BMO$ spaces are equal thus extending the classical result of Reimann and Rychener. Moreover, we show that functions in parabolic $\BMO$ are exponentially integrable in a general class of space-time cylinders. As a corollary, we establish global integrability for positive supersolutions to a wide class of doubly nonlinear parabolic equations.
\end{abstract}
  
\maketitle

\section{Introduction}
In 1993, Lindqvist \cite{Lindqvist1993} proved that positive ${ \bf A}_p$-superharmonic functions bounded away from zero are globally integrable to some small power $\epsilon > 0$ on H\"older domains, that is, on domains satisfying a quasihyperbolic boundary condition (defined in \cite{GM1985}, see also Definition \ref{def: quasihyperbolic}). This was qualitatively the most general result in a series of papers investigating global integrability of harmonic functions, starting from Armitage's result on balls \cite{Armitage1971}, and leading via Lipschitz domains \cite{MS1989} and generalizations \cite{Masumoto1992} to H\"older domains of Stegenga and Ullrich \cite{SU1995}. H\"older domains were characterized as the ones with exponentially integrable quasihyperbolic metric in \cite{SS1991}, and this was proved to imply exponential integrability of functions $\BMO$ in \cite{SS1991} and \cite{Hurri1993}. Noting the well-known fact visible in Moser's proof \cite{Moser1961} for Harnack inequality of elliptic partial differential equations, i.e.~ that $- \log u \in \BMO$ for positive supersolutions, Lindqvist composed the general result. For more about research related to global integrability, see also \cite{Suzuki1993}, \cite{Gotoh1999} and \cite{Aikawa2000}.

Contrary to the elliptic theory, in the parabolic case very little or nothing has been done, even in the case of heat equation, at least to the author's knowledge. In this paper, we prove many parabolic analogues of the above mentioned resuls. We will use a version of parabolic $\BMO$, originally introduced by Moser \cite{Moser1964}. Then, using the John-Nirenberg type inequalities known from \cite{Aimar1988}, we will establish a global version of John-Nirenberg type lemma in space-time cylinders that are H\"older domains in spatial dimensions (corresponding to the non-parabolic results from \cite{Hurri1993} and \cite{SS1991}). Once we have global John-Nirenberg inequality, we can prove the equivalence of local and global norms generalizing the classical result of Reimann and Rychener \cite{RR1975}.

As an application, we will consider equations of the form
\begin{equation}
\label{eq: equationintro}
\frac{\partial (u^{p-1})}{\partial t} = \dive  A(x,t,u,Du) ,
\end{equation}
where the function $A$ satisfies certain $p$-Laplace type growth conditions to be specified in Section 6. This class of equations has been studied for instance in \cite{Trudinger1968} and \cite{KK2007}. The key fact we will use is of course the parabolic $\BMO$-condition, and our results will apply to any equation whose solutions are exponentials of functions in parabolic $\BMO$ according to Definition \ref{def: bmo condition} or even Definition \ref{def: lag mappings}.

We conclude the introduction by briefly describing how the parabolic case differs from the elliptic one. The mean oscillation of a function $u \in L_{loc}(\mathbb{R}^{n}) $ is defined as
\[\frac{1}{\abs{B}} \int_{B} \abs{u-u_{B}} \dx,\]
with $B$ a ball, and a function is in $\BMO$ if its mean oscillation is 
uniformly bounded. In the definition of parabolic $\BMO$, there is a \textit{time lag} between the domains where we measure the upper and lower deviation from some constant. Thus the picture to have in mind about parabolic oscillation of $u \in L^{1}_{loc}(\mathbb{R}^{n+1})$ is
\begin{multline*}
\frac{1}{\abs{B \times (\theta + I)}} \int_{\theta + I} \int_{B}(u(x,t)-a_{B \times I})^{+} \dx \dt \\ +   \frac{1}{\abs{B \times( I)}} \int_{ I} \int_{B} (u(x,t)-a_{B \times I})^{-} \dx \dt,
\end{multline*} 
where $a_{B \times I}$ is a constant, $I$ an interval so that $B \times I$ respects the appropriate geometry, $\theta + I$ is the interval translated forward in time by the lag parameter $\theta > {\rm length} (I)$, and $\abs{ \cdot}$ denotes the $n+1$ dimensional Lebesgue measure. The appearance of time lag is a deep fact originating from the time lag phenomenon of parabolic partial differential equations, and it results in that parabolic $\BMO$ differs quite a lot from the classical $\BMO$.

The paper is organized as follows: Section 2 introduces notation and known results, Section 3 develops a chaining technique, Section 4 contains the proof of global John-Nirenberg inequality, Section 5 consists of the formulations of its most important consequences, and in Section 6 we apply the results to parabolic differential equations. 

\vspace{0.5cm}

\textit{Acknowledgement.} The author would like to thank Juha Kinnunen for proposing the problem and for valuable discussions on the subject.

\section{Definitions and preliminaries}
We will start with general conventions. We will work in $\mathbb{R}^{n+1}$. The first $n$-coordinates will be called \textit{spatial} (usually denoted $x$) and the last one \textit{temporal} (usually denoted $t$). We will use the standard notation $\abs{E}$ for the Lebesgue measure of $E$. In most cases we do not specify its dimension, but it must be clear from the context. When it comes to integrating, we denote $\dmu = \dx \dt$. The letter $C$ without subscript will be a constant depending only on the quantities we are not keeping track of, and we denote $f \lesssim 1$ if $f \leq C$. Occasional subscripts in this notation, such as $\lesssim_n$, will emphasize the dependencies of the constant. The positive part of a function $u$ is denoted $(u)^{+}=(u)_{+} = \chi_{\{u > 0\}}u$; the negative part is defined by $(u)^{-}= (u)_{-} = -\chi_{\{u < 0\}}u$.

\begin{figure}
\centering
\includegraphics[scale=0.7]{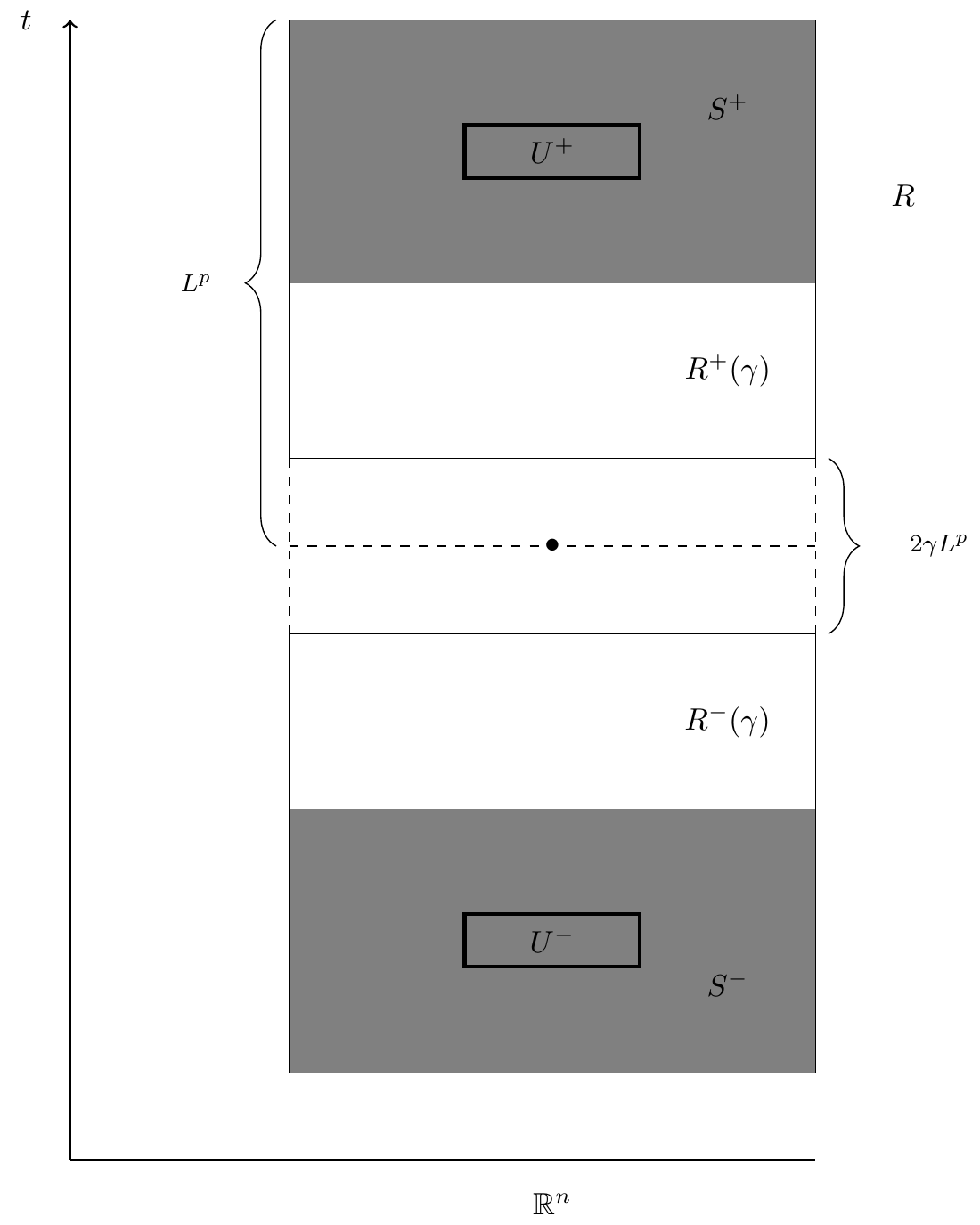}
\begin{caption}
{\label{fig:rect}} A parabolic rectangle $R$ and the arrangement of its subsets $U^{+} \subset S^{+} \subset R^{+}(\gamma)$. The lengths are not in correct scales.
\end{caption}
\end{figure}

The definition of classical $\BMO$ is stated in terms of Euclidean cubes. In the parabolic context, the class of cubes must be replaced by that of \textit{parabolic rectangles}. The notation introduced in the next definition is illustrated in Figure \ref{fig:rect}. 
\begin{definition}[Parabolic rectangle]
\label{def: parabolic rectangle}
Let $Q(x,L) \subset \mathbb{R}^{n}$ be a cube with sidelength $L$ and center $x$, and let $p> 1$ be fixed. We define a parabolic rectangle centered at $(x,t)$ with sidelength $L$, its upper half and its upper quarter as 
\begin{align*}
R &= Q \times (t-L^{p}, t + L^{p}) \\
R^{+} &= Q \times (t , t + L^{p}) \\
S^{+} &= \{(y,\tau) \in R: \tau > t+ \textstyle{\frac{1}{2}}L^{p} \} .
\end{align*}
The corresponding lower parts $R^{-}$ and $S^{-}$ are defined analogously. The \textit{parabolic scaling} of a rectangle and its quarter are defined as
\begin{align*}
\lambda R &= (\lambda Q) \times (t-(\lambda L)^{p}, t + (\lambda L)^{p})\\
\lambda S^{+} &= (\lambda Q) \times (t + \textstyle{\frac{3}{4}} L^{p} - \textstyle{\frac{1}{4}} (\lambda L)^{p}, t + \textstyle{\frac{3}{4}} L^{p} + \textstyle{\frac{1}{4}}(\lambda L)^{p})
\end{align*}
A special but technical role is played by the sets 
\[U^{+ } =   \frac{1}{8} S^{+}  ,\]
called \textit{upper fragments}. From now on, the symbols $Q$, $R^{\pm}$, $S^{\pm}$ and $U^{\pm}$ will be reserved for sets introduced in this definition.
\end{definition}
It may be useful to notice that $S^{+}$ is a metric ball with respect to 
\begin{equation}
\label{eq:parametric}
d((x,t),(y,\tau)):= \max\{\norm{x-y}_{\infty}, C_p \abs{t-\tau}^{1/p} \}.
\end{equation}
This motivates the way we define its scaling. Moreover, the centers of various $n+1$ dimensional sets mean centers of the corresponding metric balls with respect to \eqref{eq:parametric}.

The parabolic $\BMO$ spaces arise from partial differential equations, and the scaling properties of the equation determine the number $p$ specifying the shape of parabolic rectangles. Since parabolic $\BMO$ condition is stated in terms of parabolic rectangles, different values of $p$ lead to different function spaces. In what follows, $p>1$ is considered to be fixed. The following definition is different from the one given in \cite{FG1985}, but it seems to be more suitable when investigating local-to-global phenomena. 

\begin{definition}[Parabolic $\BMO$]
Let $\Omega \subset \mathbb{R}^{n}$ be a domain and $T \in (0,\infty]$. Given $\sigma \geq 1$, a function $u \in L^{1}_{loc}(\Omega \times (0,T))$ is in parabolic $\BMO$, denoted $\PBMO^{\sigma}(\Omega \times (0,T))$ if for each parabolic rectangle $R$ there is a constant $a_R$ such that
\begin{equation}
\label{def: bmo condition}
\sup_{\sigma R \subset \Omega \times (0,T)} \left( \dashint_{  S^{+}} (u-a_R)^{+} \dmu +  \dashint_{  S^{-}} (a_R- u)^{+} \dmu \right) =:\norm{u}_{\PBMO^{\sigma}}  < \infty.
\end{equation}
In case $\sigma = 1$, it will be omitted in the notation.
\end{definition}

The starting point for our considerations is the John-Nirenberg inequality satisfied by $u \in \PBMO$. In sufficient generality it was proved by Aimar \cite{Aimar1988}, who studied $\BMO$-spaces with lag maps. His aim was to develop a unified approach to both classical and parabolic $\BMO$ on spaces of homogeneous type. The next technical definition is a special case of Definition 1.4 in \cite{Aimar1988}, and it is included in order to demonstrate that Aimar's results apply to all $\BMO$-type spaces discussed in this paper.

\begin{definition}[$\BMO$ spaces with certain lag mappings]
\label{def: lag mappings}
Let $r \in \mathbb{R}$ and $b \in (0,1]$. 
For an $L$-sided parabolic rectangle $R$ centered at $(y,\tau)$, define 
\begin{align*}
T_{R}(x,t;L) &= (x,t-L^{p};L ) \\
T_{S}(x,t;L) &= \left(x,t-\frac{3}{2} L^{p};L\right) \\
h(r) &= (r)_{+}^{b}.
\end{align*}
A function $u \in L_{loc}^{1}( R^{*} )$ is said to belong to $\BMO$ space with lag mapping $T_B $ ($B=S$ or $B=R$) with respect to $h $ if for each parabolic subrectangle $R \subset R^{*}$, there is a constant $a_R$ such that 
\[\sup_{R \subset R^{*}} \left( \dashint_{  B^{+}} h(u-a_R) \dmu +  \dashint_{  B^{-}} h(a_R- u)  \dmu \right) < \infty. \]
Note that the lag maps $T_B$ take a center and a radius of a ball with respect to a metric \eqref{eq:parametric}, and map them to a new center and a new radius. It is easy to see that with these choices of $T_R$ and $T_S$, $B^{-}$ is the metric ball corresponding to the data from $T_B$ and $B^{+}$.
\end{definition}

The choice $b= 1$ and $B = S$ gives the condition \eqref{def: bmo condition}. Any space of Definition \ref{def: lag mappings} satisfies the following John-Nirenberg inequality (Lemma 3.4 in \cite{Aimar1988}) with certain modifications, but we will state it for the case of $\PBMO^{\sigma}$. Aimar's approach  allows us to use general $p> 1$ instead of $p=2$ of \cite{FG1985}.
\begin{lemma}[Aimar \cite{Aimar1988}]
Let $u \in \PBMO^{\sigma}(\Omega \times (0,T)) $. Then there are constants $A$ and $B$ depending only on $\norm{u}_{\PBMO^{\sigma}}$, $p$ and $n$ such that for each parabolic rectangle $R$ with $\sigma R \subset \Omega \times (0,T)$ the following holds:
\begin{align}
\label{JN+}
 \abs{  U^{+} \cap \{ (u-a_R)^{+} > \lambda\}} &\leq A e^{-B\lambda} \abs{U^{+}} \quad \textrm{and} \\
 \label{JN-}
 \abs{ U^{-}\cap \{ (a_R-u)^{+} > \lambda\}} &\leq A e^{-B\lambda} \abs{U^{-}} .
\end{align}
\end{lemma}

\begin{remark}
\label{remark: John-Nirenberg}
The general $\BMO$ space with lag mapping (in sense of Definition \ref{def: lag mappings}) satisfies the same inequalities but $U^{\pm}$ from Definition \ref{def: parabolic rectangle} are replaced by $\frac{1}{8}B^{\pm}$ from Definition \ref{def: lag mappings}. This difference is not essential, and all the following arguments will work also in that case. The factor $1/8$ is small enough to make the inequalities \eqref{JN+} and \eqref{JN-} hold, but its role is not important in this paper. In fact, our final result in $\mathbb{R}^{n+1}$ will make the technical notion of fragments unnecessary. 
\end{remark}

Finally, we will need a geometric condition that ensures that the domains we consider are reasonable enough. The following class of (bounded) domains was first defined by Gehring and Martio in \cite{GM1985}.
\begin{definition}[Quasihyperbolic boundary condition]
\label{def: quasihyperbolic}
Let $\Omega \subset \mathbb{R}^{n}$ be a domain. We define its quasihyperbolic metric as 
\[k(x,y) := \inf_{\gamma_{xy}} \int_{\gamma_{xy}} \frac{1}{d(z, \Omega^{c})} \, {\rm d} s(z), \]
where the infimum is over curves connecting $x$ and $y$. A domain is said to satisfy a \textit{quasihyperbolic boundary condition} if there is a fixed $x_0 \in \Omega$ and a constant $K$ such that for all $y \in \Omega$
\[k(x_0,y) \leq K \log \frac{K}{d(y,\Omega^{c})} .\]
\end{definition}

\section{Chain lemma}
When proving a local-to-global result for a function $u$ on some domain, the most crucial part is to get information about the behavior of $u$ close to the boundary of the domain. We will use a chaining technique composing various ideas from the known results in \cite{Maasalo2008}, \cite{Staples2006} and \cite{Buckley1999}. The original chaining techinques do not work as such since the parabolic $\BMO$ condition compells us to take into account the special role played by the time variable. 

The new problem is that the rectangles in the chain cannot be located as freely as they could in the classical case. The chain has to have a direction in time, each step in spatial dimensions forcing us to take certain step in time. Thus the first challenge is to ensure that for each point in the space-time cylinder there is enough time to move to the spatial point we consider the center. This problem can be solved by imposing an artificial upper bound on the size of spatial steps, which is reflected as the $p$th power to the time variable due to the parabolic scaling. 

\begin{lemma}
\label{lemma: chain}
Let $\Omega \subset \mathbb{R}^{n}$ be a domain and $\Delta = \Omega \times (0,T)$, $0 < \beta < 1$; $\alpha,\alpha', \delta > 0$ and let $U_{(x,t)}^{+}$ be an upper fragment of a parabolic rectangle centered at $(x,t) \in \Omega \times (\delta ,T)$ with spatial sidelength $l_{x,t}'= \min \{l_{x,t},\alpha' q \}$. Here
\begin{equation}
\label{eq: ch lemma choice condition}
l_{x,t} :=  \min \{\beta d(x,\Omega^{c}), \beta (T-t)^{1/p}, \alpha q\},
\end{equation}
where $q = \sup_{x \in \Omega} \textrm{length}(\gamma_{xz})$ is the maximal length of quasihyperbolic geodesics connecting points $x$ to a fixed $z \in \Omega$. 

Under these assumptions, the parameter $\alpha$ can be chosen so that there is a chain of parabolic rectangles $\mathcal{P}(U_x^{+}) = \{R_i\}_{i=1}^{k_{(x,t)}}$ with the following properties: 
\begin{enumerate}
\item[(i)] $R_{1} = R_{(x,t)} $, $R_{k_{(x,t)}}$ is centered at $(z,\tau(x,t))$ and it has spatial sidelength $l_{z,\tau(x,t)}$. For all $j$ we have that $\beta^{-1} R_j \subset \Delta $.
\item[(ii)] $ \abs{U_i^{-} \cap U^{+}_{i+1}} \gtrsim_{p, n, \beta} \max\{ \abs{R_{i}}, \abs{R_{i+1}} \}$ as $1\leq i < k_x$.
\item[(iii)] $0 \leq t- \tau(t,x)  \leq  q^{p}  \eta  $, where $\eta \eqsim_{n,\beta} \alpha$.
\item[(iv)] $k_{(x,t)} \lesssim_{p,n, \beta  } k(x,z) + \log \frac{T}{T-t}+ \log \left( \frac{\alpha}{\alpha'} + 1 \right) + \frac{1}{\alpha} +1 $.
\end{enumerate} 
\end{lemma}
\begin{proof}
We start by assuming that $\alpha' q \geq l_{x,t}$. The first rectangle $R_1$ is obviously given. We denote its center $p_1$, and we denote the center of its upper fragment $p_1'$. Suppose that $R_j$ (centered at $p_j = (y_j,t_j)$) has been chosen. We connect its spatial center $y_j$ to $z$ with the quasihyperbolic geodesic $\gamma$ and find the point $y_{j+1}$ where $\gamma$ exits $Q_j = Q(y_j,l_{y_j,t_j})$. Set 
\begin{align*}
t_{j+1}' &= t_{j} - \frac{3}{4}l_{j}^{p} \\
p_{j+1}' &= (y_{j+1},t_{j+1}') \\ 
l_{j+1} &= l_{(y_{j+1},t'_{j+1})} \\
t_{j+1} &= t_{j+1}' - \frac{3}{4}l_{j+1}^{p} = t_j-\frac{3}{4}(l_j^{p} + l_{j+1}^{p})\\
p_{j+1} &= (y_{j+1},t_{j+1}).
\end{align*}
We define $R_{j+1}$ by extending the spatial cube $Q_{j+1}$ to a parabolic rectangle $R_{j+1}$ centered at $p_{j+1}$ with sidelength $l_{j+1}$ so that its upper fragment is centered at $p_{j+1}'$, which is also the temporal center of $U_j^{-}$ (see Figure \ref{fig:chl}). One of the two consecutive fragments ($U_j^{-}$ and $U_{j+1}^{+}$) has its temporal projection contained in the other, and $y_{j+1} \in \partial Q(y_j,l_j)$, so in order to establish (ii), it suffices to prove that $l_j \eqsim_{\beta,n} l_{j+1}$. 

\begin{figure}
\centering
\includegraphics[scale=0.6]{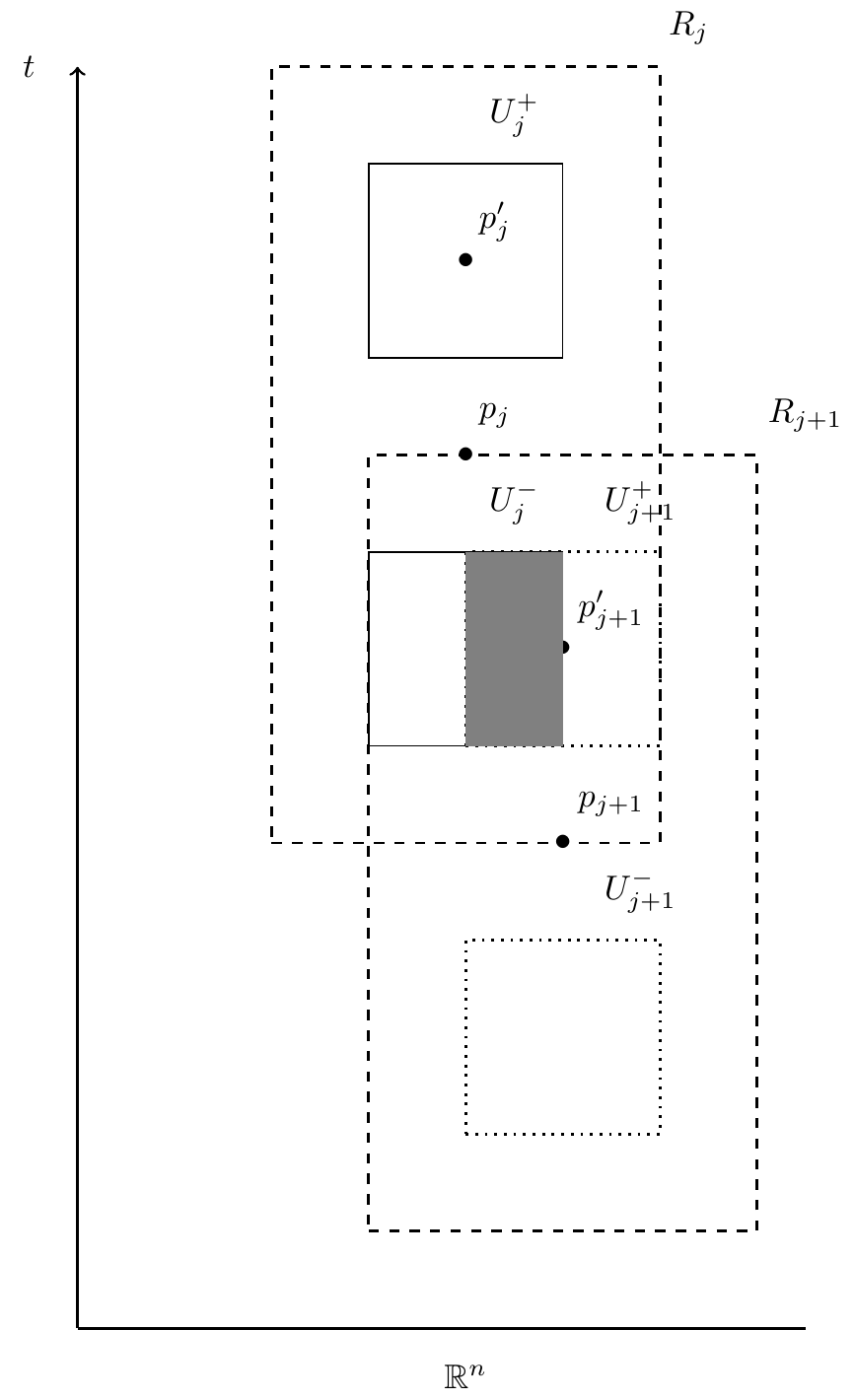}
\begin{caption}
{\label{fig:chl}Schematic picture on position of two subsequent rectangles in the chain. For clarity, the lengths and distances are not in scale.}
\end{caption}
\end{figure}

We define an auxiliary metric 
\[d'((x_1,t_1),(x_2,t_2)) = \max\{ \abs{x_1-x_2} , \abs{t_1-t_2}^{1/p} \}.\] 
Denoting $\Xi = \Omega^{c} \times \mathbb{R} \cup \Omega \times (T,\infty)$, the choice condition \eqref{eq: ch lemma choice condition} can be reformulated as 
\[l_{x,t} = \min \{\beta d'((x,t), \Xi), \alpha q \}.\] 
If $l_{i} = \alpha q  = l_{i+1}$, there is nothing to prove, so assume  that both $l_{i} = \beta d' (p_i', \Xi )$ and $l_{i+1 } = \beta d' (p_{i+1}',\Xi )$. Then
\begin{equation}
\label{eq: ch lemma 11}
l_i \leq \beta( d'(p_i',p_{i+1}') + d'(p_{i+1}', \Xi)) \leq \beta l_i + l_{i+1}  
\end{equation}
and 
\begin{equation}
\label{eq: ch lemma 22}
l_{i+1} \leq \beta( d'(p_i',p_{i+1}') + d'(p_{i}', \Xi))  \leq (\beta + 1) l_i .
\end{equation}
If, in turn,  
\[l_{i+1} = \beta  d' (p_{i+1}', \Xi ) \leq \alpha q   = l_{i},\]
then one direction is clear, and for the other, \eqref{eq: ch lemma 11} still holds. The last alternative $l_{i} = \beta  d '(p_{i }', \Xi ) \leq \alpha   q = l_{i+1}$ is done similarly by  \eqref{eq: ch lemma 22}. Thus (ii) holds.

For (iii), note that since 
\[\sum_{i} l_i \leq N  q, \]
where $N$ depends only on $n $ and $\beta$, the choice $\alpha^{p-1} \leq \frac{\eta }{2 N} $ yields the same bound for all $(l_i/q)^{p-1}$, and consequently
\[t- \tau(x,t) \leq \sum_{i} 2l_i^{p} \leq \frac{\eta q^{p-1}  }{N} \sum_{j} l_j \leq \eta q^{p}. \]

To prove (iv), assume first that $\alpha$ does not affect the chain length. It is straightforward to check that for all pairs of consecutive indices $(i,i+1)$ where both sidelengths $l_i$ are not determined by the temporal dimension (that is, one or both of them are determined by the spatial dimensions); $l_i$, $l_{i+1}$, $d(x,\Omega^{c})$ and $\abs{\gamma_{y_iy_{i+1}}}$ are all comparable for $x \in \gamma_{y_iy_{i+1}}$. We call subchains with this kind of center points $p_i$ \textit{mixed} and label the corresponding segments of $\gamma$ by $\gamma_{ i_j,i_j +1  }$. Then
\begin{align*}
k(x,z) &= \int_{\gamma} \frac{1}{d(y,\Omega^{c})} \, {\rm d}s(y) = \sum_{i}\int_{\gamma_{y_iy_{i+1}}} \frac{1}{d(y,\Omega^{c})} \, {\rm d}s(y) \\
&\geq \sum_{j=1}^{k_x} \int_{\gamma_{i_j,i_j +1}} \frac{1}{d(y,\Omega^{c})} \, {\rm d}s(y)   \gtrsim k_x. 
\end{align*}
Multiplying $k_x$ by $2$, we may assume that it controls the actual number of points $p_i$ in mixed subchains.

It remains to estimate the gaps between mixed subchains. These are filled by purely temporal subchains, where the sidelengths of consecutive rectangles are determined by the distances from $\Omega \times \{T\}$. Here the fact $t_{i+1}' = t_{i }  -  \frac{3}{4}l_i^{p} = t_i'- \frac{3}{2}l_i^{p}$  implies 
\begin{align*}
l_{i+1}^{p} &= \beta^{p}( T - t_{i+1}') = \beta^{p}(T - t_i') + \frac{3}{2} \beta^{p} l_i^{p} =  \left(1+ \frac{3}{2}\beta ^{p} \right) l_i^{p} =: Ml_i^{p}.
\end{align*}
Given two temporal subchains with no temporal subchain in between, the successor will start with a rectangle greater than the ending rectangle of the predecessor. Moreover, the last rectangle in a purely temporal subchain is also a starting rectangle for the following mixed subchain. By the convention on $k_x$, these will be counted to both mixed and temporal chains. Without decreasing the actual number of counted rectangles, we may join the temporal subchains by replacing the starting rectangles of the purely temporal subchains by the ending rectangles of the predecessors.

This new chain of rectangles with centers $\{(y_{i_\iota},t_{i_\iota})\}_{\iota =0}^{k}$ satisfies
\begin{align*}
T \gtrsim \sum_{\iota =0}^{k} 2 l_{i_\iota}^{p} \geq 2l_{i_0}^{p} \frac{M^{k+1}-1}{M-1} \gtrsim C l_{i_0}^{p}M^{k - C} 
\end{align*}
where all dependencies are on $\beta$. Especially
\[k \lesssim \log \frac{T}{l_{i_0}^{p}} + 1.\]
Adding now the worst contribution of $\alpha q$-sided rectangles, we have established
\[k_{x,t} \lesssim k(x,z) + \log \frac{T}{T-t} + \frac{1}{\alpha} + 1.\]
Up to the additional assumption on $\alpha' \geq l_{x,t}$, this is (iv).

To get rid of the assumption $\alpha'q \geq l_{x,y}$, we just start the construction by doubling the sidelength at each step until the choice condition of $l_{y,t}$ becomes active. It can be checked that this does not affect other bounds than the number of rectangles, and here the effect is at worst the claimed $\log (\alpha /\alpha' + 1) +1$.
\end{proof}

\begin{remark}
\label{re: vertivcal}
Given two parabolic rectangles $R= Q \times (t-L^{p}, t + L^{p})$ and $R'= Q \times (t'-L^{p}, t'+L^{p})$ such that $T \geq t'-t \geq M  L$ for some big $M$ (say $M \geq  100$), then $R'$ can be connected to $R$ with a chain $\{R_i\}_{i=1}^{k}$ satisfying (ii) of Lemma \ref{lemma: chain} (with dimensional constant) and $k \leq CM$ where $C$ is a numerical constant. This is practically done by looking at chains constructed as in Lemma \ref{lemma: chain}, but using a constant spatial cube $Q$, and choosing the midpoint of $R_{i+1}$ asking only that  $\abs{{\prt R_{i} \cap R_{i+1} }} \geq \frac{L^{p}}{M}$ is satisfied ($\prt$ means the projection on the temporal variable). This gives us flexibility to squeeze or stretch the chain in order to synchronize the endpoint rectangles provided by the previous lemma.
\end{remark}

\section{A global John-Nirenberg inequality}
In this section we will prove one of the main results of this paper, the global John-Nirenberg inequality. 

\begin{theorem}
\label{thm: john-nirenberg}
Let $\Omega$ satisfy a quasihyperbolic boundary condition. If $u$ is a function in $ \PBMO^{\sigma} ( \Omega \times (0,T))$, then for every $\delta  > 0$ there is $c \in \mathbb{R}$ and constants $A$ and $B$ depending on $\delta$, $\sigma$, $p$, $n$, $\norm{u}_{\PBMO^{\sigma}}$ and the data of $\Omega$, such that
\[\abs{  \Omega \times (\delta q^{p} , T) \cap \{ (u-c)^{+} > \lambda\}} \leq A e^{-B\lambda} \abs{\Omega \times (\delta q^{p} , T)}.\]
Here $q$ is again the maximal length of quasihyperbolic geodesics.
\end{theorem}

\begin{proof}
Choose $ \sigma < \beta^{-1}  $ and let $\alpha>0$ be a constant to be determined later. For each $y \in \Omega$, let $5l_y = \min \{\beta d(y,\Omega^{c}), \alpha q \}$ and denote $Q_y = Q(y,l_y)$. Using 5-covering lemma, we may extract a countable collection $\mathcal{W}_{\alpha} =  \{Q_i\}_{i} :=\{ 5 Q_{y_i}\}_{i} $ so that the cubes $\{Q_{i}\}_{i}$ cover $\Omega$ and $\frac{1}{5}Q_i$ are pairwise disjoint. Moreover, we may ask $Q_1$ to be centered at $z$, the distinguished point of $\Omega$. The symbols $\mathcal{W}_\zeta$ will refer to similar constructions with additional size bounds $5l_y \leq \zeta q$. We denote $\delta_0 = \delta  q^{p}$.

First look at a fixed time level $\Omega \times \{ \delta_0 \} $. We extend every cube $Q_i$ to be an upper fragment of a parabolic rectangle $R_i$ having its lower face on $\Omega \times\{\delta \}$. Using Lemma \ref{lemma: chain}, we may construct chains $\mathcal{P}(U_i^{+})$ connecting these $U_i^{+}$, upper fragments of $R_i$, to rectangles with spatial projections coinciding with $Q_1$. These rectangles may, however, be centered somewhere in $\Omega  \times (-\infty, 0)$, since in the construction of the parabolic chain, connecting the lower fragments of rectangles with the upper fragments of their successors makes the chain flow down to the past. To deal with this, we choose $\eta$ in (iii) of Lemma \ref{lemma: chain} to be $(10\sigma)^{-10}  \delta  $ so that the final rectangle will definitely be in $\Omega \times ( \frac{1}{2} \delta_0   ,T )  $ and admissible in the definition of $\PBMO^{\sigma}$. This imposes an upper bound $b$ on $\alpha$ (depending only on $\delta$ and $\sigma$). 

Next we slice the cylinder $\Omega \times (\delta_0, T)$ both spatially and temporally. We begin with the time. Denote 
\begin{align*}
\tau_j &= T - 2^{-j} (T- \delta_0), \  \textrm{as}  \ j \geq 0 \quad  \textrm{and}  \quad \\
Z_j &= \Omega \times (\tau_j, \tau_{j+2} ).
\end{align*}
The union of these $Z_j$ is included in $\Omega \times (\delta_0, T)$ and their overlap is bounded by $2$. When it comes to space, we define \[\Omega_k = \Omega \cap \{ d(x,\Omega^{c}) < 2^{-k}   \}.\] 
This lets us partition $\mathcal{W}_{\alpha} $ so that \[\mathcal{W}_{\alpha}^{k} = \{Q_x \in \mathcal{W}_{\alpha}: x \in \Omega_k \setminus \Omega_{k+1} \}.\] 
Since for each $y \in Q(x,l) \in \mathcal{W}_{\alpha}^{k}$ we have that
\[d(y,\Omega^{c}) \leq \abs{x-y} + d(x,\Omega^{c}) \leq ( \beta + 1 ) d(y,\Omega^{c}) \leq  2^{-k+1}, \]
the inclusion 
\[\bigcup_{Q \in \mathcal{W}_{\alpha}^{k}} Q \subset \Omega_{k-1}\]
will follow.

Then we cover $\Omega \times (\tau_j, \tau_{j+1} )$. These subsets of $Z_j$ will in turn cover the whole space-time cylinder $\Omega \times (\delta_0,T)$. Choose 
\[\alpha_j = \min \left\lbrace    b, \frac{\beta}{q}  \left( \frac{T-\delta_0}{ 2^{j+2}}\right)^{1/p} \right\rbrace  .\] 
Take the cover $\mathcal{W}_{\alpha_j}$, extend its cubes to upper fragments of parabolic rectangles $U_i^{j+}$. At each spatial $Q_i$, stack these parabolic fragments pairwise disjointly minimal amount to cover the temporal interval $(\tau_j, \tau_{j+1} )$. At the future end, the stack will not exceed $\tau_{j+2}$. Label the fragments as $\mathcal{Z}_j = \{U_i^{jk+}\}_{ijk}$. 

For each $U_{i}^{jk+}$ we form the chain of Lemma \ref{lemma: chain} with $\alpha = b$ and $\alpha' = \alpha_j$. It has $m_i^{j}$ rectangles. We want, however, to make all final rectangles coincide not only spatially but also temporally at, say, $\mathfrak{R}$. In order to do that, we must continue the chain of $U_{i}^{jk}$ with $m_j$ rectangles. According to Remark \ref{re: vertivcal}, recalling the choice of $\eta $, and carefully checking the interdependence of $j$ and temporal slicing of $\Omega \times (\delta_0,T)$, we see that this can be done with bound
\[m_j \lesssim \sum_{\iota=0}^{j} \frac{2^{-j}(T-\delta_0)}{(\alpha_j q)^{p}} \lesssim  \frac{j}{\delta}   .\]
Call these continued chains $\mathcal{C}(U_{i}^{jk+})$. 

Now we are in position to prove the claim. Recall that $a_R$ is always a constant from the $\PBMO^{\sigma}$ condition \eqref{def: bmo condition}. Let $C_0$ be the constant in part (ii) of Lemma \ref{lemma: chain}. For a while, we denote $\mathcal{C}(U_{i}^{jk+}) = \{P_\iota\}_{\iota=1}^{m_{i}^{j} + m_j}$. We will use an argument from \cite{Staples2006}. By a suitable choice of $\lambda_0$ (depending only on $\norm{u}_{\PBMO^{\sigma}}$, $\beta$ and the dimension) in John-Nirenberg inequalities \eqref{JN+} and \eqref{JN-}, we get  
\begin{align*}
\abs{E_{\lambda_0,\iota}^{-}} := \abs{ U_\iota^{-} \cap \{ (a_{P_\iota}-u)^{+} > \lambda_0\}} &\leq  \frac{C_0}{2} \abs{U_{\iota}^{-}}, \\
\abs{E_{\lambda_0,\iota + 1}^{+}} := \abs{  U_{\iota+1}^{+}\cap \{ (u-a_{P_{\iota  +1 }})^{+} > \lambda_0\}} &\leq  \frac{C_0}{2} \abs{U_{\iota + 1}^{+}}, \\
\end{align*} 
and
\[\abs{ (U_\iota^{-} \cap U_{\iota + 1}^{+}) \setminus (E_{\lambda_0,\iota}^{-} \cup  E_{\lambda_0,\iota + 1}^{+})} > 0. \]
This indicates that there is $p_\iota \in ( U_\iota^{-} \cap U_{\iota + 1}^{+}) \setminus (E_{\lambda_0,\iota}^{-} \cup  E_{\lambda_0,\iota + 1}^{+})$, and consequently
\begin{align*}
(a_{P_1}-a_{\mathfrak{R}})^{+} &\leq  \sum _{\iota =1}^{m_i^{j}+m_j-1} (a_{P_\iota}-a_{P_{\iota + 1}})^{+} \\
&\leq  \sum _{\iota =1}^{m_i^{j}+m_j-1} (a_{P_\iota}-  u(p_\iota))^{+} + (u(p_\iota)-a_{P_{\iota + 1}})^{+}  \\
&\lesssim_{\lambda_0}   m_i^{j} + m_j.
\end{align*}

Now for every $\lambda > 0$, we get
\begin{align}
\label{equ: SS final}
 \abs{U_i^{jk+} \cap & \{ (u-a_{\mathfrak{R} })^{+} > \lambda \}} \nonumber \\
& \leq    \abs{U_i^{jk+}  \cap \{(u-a_{R_{i}^{jk}})^{+} +  (a_{P_{\iota} } -a_{\mathfrak{R} })^{+} > \lambda\}} \nonumber \\
& \lesssim   \abs{U_i^{jk+}  \cap \{(u-a_{R_j^{jk} })^{+} > \lambda /2 \}} \nonumber \\ 
& \hspace{2cm} +  \abs{U_i^{jk+} \cap \{ C(m_i^{j}+m_j)   > \lambda/2 \}} .
\end{align}
By John-Nirenberg, the first term can be estimated by
\begin{equation*}
\abs{U_i^{jk}  \cap \{(u-a_{R_i^{jk} })^{+} > \lambda /2 \}} \lesssim \abs{U_i^{jk}} e^{-B\lambda}.
\end{equation*}
Moreover
\[\sum_{i,j,k} \abs{U_i^{jk} }= \sum_{j} \sum_{U \in \mathcal{Z}_j } \abs{U_i^{jk}}  \lesssim \sum_{j} \abs{Z_j} \lesssim \abs{\Omega \times (\delta_0 , T)}, \]
so
\begin{equation}
\label{equ SS john1}
\sum_{i,j,k} \abs{U_i^{jk}  \cap \{(u-a_{R_i^{jk} })^{+} > \lambda /2 \}} \lesssim e^{-B\lambda} \abs{\Omega \times (\delta_0 , T)}.
\end{equation}

We then turn to the second one. Since $U_i^{jk} \subset Z_j$, by Lemma \ref{lemma: chain} we have that 
\begin{equation}
\label{eq q-dependence}
m_i^{j} + m_j \lesssim_\alpha k(y_{i}, z) + \frac{j}{\delta}+ j \log \frac{q}{(T-\delta_0)^{1/p}}  + 1.
\end{equation}
Denote 
\begin{align*}
&  \abs{U_i^{jk+} \cap \{ k(y_i,z) + 1  > C \lambda  \}}  +  \abs{U_i^{jk+} \cap \{   j     > C \lambda  \}}   \\
 &= I_{ijk} + II_{ijk}  .
\end{align*}
Since $\Omega$ satisfies a quasihyperbolic boundary condition, it holds that $\abs{\Omega_{k}} \lesssim 2^{-\nu k}  \abs{\Omega}$ for some $\nu > 0$ (see \cite{Buckley1999} and \cite{KR1997}). On the other hand, the quasihyperbolic boundary condition itself,
\[k(y_i,z)  \lesssim \log \frac{K}{d(y_i,z)},\]
gives that if $Q_i \in \mathcal{W}_{\alpha_j}^{k}$, then $k(y_i,z) \lesssim 1 + k$. Denoting $t_i ^{jk} = \abs{\prt U_{i}^{jk+}}$, we may compute (since $Q_i$ have bounded overlap)
\begin{align}
\label{equ SS john2}
\sum_{i,j,k} I_{ijk} &\leq   \sum_{i,j,k} t_i^{jk} \cdot \abs{Q_i \cap \{k(y_i,z) +1 > C\lambda \}}   \nonumber \\
&\lesssim (T- \delta_0)  \sum_{l> C\lambda - C'}  \abs{\Omega_{l}}   \nonumber  \\
&\lesssim (T- \delta_0)  \sum_{l> C\lambda - C'} 2^{-\nu l}  \abs{\Omega} \nonumber \\
&\lesssim  \abs{\Omega \times ( \delta_0 , T)}  e^{-B\lambda}.
\end{align}

For the term $II_{ijk}$, we note that since $\frac{1}{5}Q_i$ are pairwise disjoint 
\begin{align}
\label{equ SS john3}
\sum_{i,j,k} II_{ijk} &=  \sum_{ j> C  \lambda  } \sum_{i,k}    \abs{U_i^{jk+} }    \lesssim    \sum_{ j> C   \lambda }   \abs{Z_j}  \nonumber \\
& \lesssim     \sum_{ j> C \lambda  } 2^{-j} \abs{\Omega \times (\delta_0, T)}  \nonumber \\
& \lesssim \abs{\Omega \times (\delta_0, T)}  e^{-B\lambda}.
\end{align}
Summing over $i,j$ and $k$ in \eqref{equ: SS final}, and plugging in the estimates \eqref{equ SS john1}, \eqref{equ SS john2} and \eqref{equ SS john3}, we get the claimed
\[\abs{  \Omega \times (\delta q^{p} , T) \cap \{ (u-a_{\mathfrak{R}})^{+} > \lambda\}} \leq A e^{-B\lambda} \abs{\Omega \times (\delta q^{p} , T)}.\]
\end{proof}

Note that if $\Omega = Q$ is a Euclidean cube and $T = 2l(Q)^{p}$, we have a parabolic rectangle. In this simple case, one may compute the bounds coming from the quasihyperbolic boundary condition explicitly (compare to \cite{Maasalo2008}). Replacing the quasihyperbolic geodesics by straight lines, and repeating the previous proof, it will be clear that the dependence on the data of $\Omega$ will become a dimensional constant. Thus we get a slightly stronger statement for these domains.




\begin{corollary}
\label{cor: john-rectangle}
Let $R$ be a parabolic rectangle centered at $(x,\tau)$. If $u$ is in $ \PBMO^{\sigma} (R)$, then for every $\delta  > 0$ there is $c \in \mathbb{R}$ and constants $A$ and $B$ depending on $\delta$, $\PBMO^{\sigma}$-norm of $u$, $\sigma$, $p$ and $n$ such that
\[\abs{R_\delta^{+} \cap \{ (u-c)^{+} > \lambda\}} \leq A e^{-B\lambda} \abs{R}.\]
Here $R_\delta^{+} = Q \times (\tau - ( 1 - \delta) L^{p}, \tau + L^{p})$.
\end{corollary}

\section{Consequences of the global inequality}
In the classical context, the global John-Nirenberg inequality can be regarded as the strongest local-to-global result of $\BMO$. By this we mean that most other results can be deduced directly from it. In this section we state and prove parabolic analogues of the exponential integrability of functions in $\BMO$ (see \cite{SS1991} and \cite{Hurri1993}) and the equivalence of local and global norms (see \cite{RR1975}).

\begin{theorem}
\label{thm: ekspint}
Let $\Omega \subset \mathbb{R}^{n}$ satisfy the quasihyperbolic boundary condition \ref{def: quasihyperbolic}. If $u \in \PBMO_{\sigma} ( \Omega \times (0,T))$ with $\sigma \geq 1$, then for every $\delta > 0$ there is $\gamma  >0 $ and $c \in \mathbb{R}$ such that
\[\int_{\Omega \times (\delta, T)} e^{\gamma (u-c)^{+}} \dmu < \infty. \]
\end{theorem} 
\begin{proof}
We may write
\begin{align*}
\int_{\Omega \times (\delta, T)}& e^{\gamma (u-c)^{+}} \dmu = \int_0^{\infty} \abs{ \{\Omega \times (\delta , T): e^{\gamma (u-c)^{+}} > \nu \}} \, {\rm d} \nu \\
&= \abs{\Omega \times (\delta, T)} +   \int_1^{\infty} \abs{ \{\Omega \times (\delta , T): e^{\gamma (u-c)^{+}} > \nu \}} \, {\rm d} \nu,
\end{align*}
so it suffices to estimate the second term. By theorem \ref{thm: john-nirenberg} we have
\[\abs{ \{\Omega \times (\delta , T): (u-c)^{+} > \lambda\}} \leq A e^{-B\lambda} \abs{\Omega \times (\delta , T)}.\] 
Using this, we get
\begin{align*}
\int_1^{\infty} &\abs{ \{\Omega \times (\delta , T): e^{\gamma (u-c)^{+}} > \nu \}} \, {\rm d} \nu \\
& = \int_{0}^{\infty} e^{\lambda} \abs{ \{\Omega \times (\delta , T):  (u-c)^{+}  > \lambda / \gamma \}} \dla \\
& \leq A \abs{\Omega \times (\delta , T)} \int_{0}^{\infty} e^{\lambda(1-B/\gamma)} \dla.
\end{align*}
Taking $\gamma$ small enough, we see that this integral is finite.
\end{proof}

Of course, we have a corresponding result for the negative part of the function. Here the gap between the domain of integration and the temporal boundary will be at the positive end.
\begin{corollary} 
\label{cor: ekspint}
Let $\Omega \subset \mathbb{R}^{n}$ satisfy the quasihyperbolic boundary condition \ref{def: quasihyperbolic}. If $u \in \PBMO_{\sigma} ( \Omega \times (0,T))$ with $\sigma \geq 1$, then for every $\delta > 0$ there is $\gamma  >0 $ and $c \in \mathbb{R}$ such that
\[\int_{\Omega \times (0, T- \delta  )} e^{\gamma (u-c)^{-}} \dmu < \infty. \]
\end{corollary}
\begin{proof}
From the proof of Theorem \ref{thm: ekspint} it is clear that once we have 
\begin{equation}
\label{eq: ekspint proof}
\abs{ \{\Omega \times (0 , T-\delta): (u-c)^{-} > \lambda\}} \leq A e^{-B\lambda} \abs{\Omega \times (\delta , T)},
\end{equation}
the claim will follow. In the proof of Theorem \ref{thm: john-nirenberg} we used Lemma \ref{lemma: chain}. It gave chains $\{R_i\}_i$ where $U_{i}^{-} \cap U^{+}_{i+1}$ had large measure. The same construction could have been done to the reversed direction, that is, so that $U_{i}^{+} \cap U^{-}_{i+1}$ would have been a large set. Repeating the proof of Theorem \ref{thm: john-nirenberg} with this orientation, we get \eqref{eq: ekspint proof}, and we are done.
\end{proof}

The next result, originally due to Reimann and Rychener \cite{RR1975} (see also Staples \cite{Staples1989}), tells that even if we originally assume the $\PBMO^{\sigma}(\Omega)$ condition with $\sigma > 1$ we actually have the condition with $\sigma = 1$. In other words, even if our original assumption is absolutely local, we still have complete information about the behaviour of a function up to the boundary.

\begin{theorem}
\label{thm: reimann_rychener}
Let $\Omega \subset \mathbb{R}^{n+1}$ be an arbitrary domain and let $\sigma > 1$. Let $u \in \PBMO_{\sigma} ( \Omega )$. Then $u \in \PBMO  ( \Omega )$.
\end{theorem}
\begin{proof}
Take a parabolic rectangle $R \subset \Omega$. By Corollary \ref{cor: john-rectangle}, we have
\[\abs{R_\delta^{+} \cap \{ (u-c)^{+} > \lambda\}} \leq A e^{-B\lambda} \abs{R}.\]
Integrating this, we get
\begin{align*}
\int_{R_{\delta}}  (u-c)^{+} \dx &= \int_{0}^{\infty} \abs{R_\delta^{+} \cap \{ (u-c)^{+} > \lambda\}} \dla \\
& \leq \abs{R} \int_{0}^{\infty} e^{-B\lambda} \dla \lesssim \abs{R}.
\end{align*}
Reasoning as in the previous proof, we see that the corresponding inequality holds for $R_\delta^{-}$. Moreover, choosing $\delta = 5/4$ and making the final rectangles $\mathfrak{R}$ in the proof of \ref{thm: john-nirenberg} coincide, we ensure that the constants $c$ associated to plus and minus parts coincide. 
\end{proof}

\section{Integrability of supersolutions}
In this section we apply the results about $\PBMO$ to partial differential equations. More precicely we study equations of the form
\begin{equation}
\label{eq: equation}
\frac{\partial (u^{p-1})}{\partial t} = \dive  A(x,t,u,Du) , \quad 1<p< \infty
\end{equation}
where $A(x,t,u,Du)$ is a Caratheodory function (see \cite{DGV2012}) satisfying the growth conditions
\begin{align}
\label{eq:growth1}
A(x,t,u,Du) \cdot Du &\geq C_0 \abs{Du}^{p} \\
\label{eq:growth2}
\abs{A(x,t,u,Du)}  &\leq C_1 \abs{Du}^{p-1}.
\end{align}

We denote by $L^{p}(0,T;W^{1,p}(\Omega))$ the space of $p$-integrable functions on $(0,T)$ having their values in the Sobolev space $W^{1,p}(\Omega)$. More concretely, $u$ is in the parabolic space $ L^{p}(0,T;W^{1,p}(\Omega))$ if 
\[u(t,\cdot) \in W^{1,p}(\Omega) \ \textrm{for a.e.} \ t \in (0,T) \] and 
\[\int_{0}^{T} \norm{u(t,\cdot)}_{W^{1,p}(\Omega)}^{p} \dt < \infty. \]

\begin{definition}
A function $u \in L^{p}_{loc}(0,T;W_{loc}^{1,p}(\Omega))$ is a \textit{supersolution} to \eqref{eq: equation} if
\[\int_{0}^{T} \int_{\Omega} \left(A(x,t,u,Du) \cdot D\phi - u^{p-1} \frac{\partial \phi}{\partial t} \right) \dx \dt \geq 0 \]
for all non-negative $\phi \in C_c^{\infty}(\Omega \times (0,T))$.
\end{definition} 
\noindent If the integral in the above definition vanishes for all $\phi \in C_c^{\infty}(\Omega \times (0,T))$, then $u$ is a weak solution. For our purposes, however, it suffices to consider supersolutions. 

Following \cite{KK2007} (see also \cite{Trudinger1968} and \cite{Moser1964}) one can show that a positive supersolution $f$ of \eqref{eq: equation} that is bounded away from zero has the negative of its logarithm in $\PBMO$. This will imply that $f$ is globally integrable to some small power $\epsilon > 0$,
which is a delicate fact since both the equation \eqref{eq: equation} and the definition of supersolution are very local assumptions, that is, they do not say anything about the behavior of $f$ near $\partial\Omega \times (0,T)$. But still, in terms of integrability, $f$ behaves at worst as a power function.

We will use Lemma 6.1 of \cite{KK2007}, which is stated for positive supersolutions of doubly nonlinear equation, that is \eqref{eq: equation} with $A(x,t,Du,u) = \abs{Du}^{p-2}Du$, but as the authors of \cite{KK2007} mention, the assumptions that the proof actually requires are the conditions \eqref{eq:growth1} and \eqref{eq:growth2}.

\begin{lemma}[Kinnunen-Kuusi \cite{KK2007}]
\label{lemma:KinnunenKuusi}
Let $f > \gamma > 0$ be a supersolution to \eqref{eq: equation} on $\sigma R $ where $\sigma > 1$ and $R$ is a parabolic rectangle. Then there are constants $C$ and $C'$ depending only on $C_0$, $C_1$, $\sigma$, $p$ and $n$ such that
\begin{align*}
\abs{\{(x,t) \in  R^{-} : \log f > \lambda + \beta + C' \} } &\leq \frac{C}{\lambda^{p-1}} \abs{ R^{-}} \quad \textrm{and} \\
\abs{\{(x,t) \in  R^{+} : \log f < -\lambda + \beta - C' \} } &\leq \frac{C}{\lambda^{p-1}} \abs{ R^{+}}   
\end{align*}
where $\beta$ depends on $R$ and $f$, and $\lambda >0$ is arbitrary.
\end{lemma}  

A short calculation shows that this estimate leads to the parabolic $\BMO$-space similar to the one first defined in Moser \cite{Moser1964}. The difference here is again the $p$-scaling of time variable. However, having the John-Nirenberg type inequalitites of \cite{Aimar1988} and Theorem \ref{thm: reimann_rychener}, we will be able to prove that  $- \log f \in \PBMO$. 

\begin{lemma}
\label{lemma: logbmo}
Let $f > \gamma > 0$ be a supersolution to \eqref{eq: equation} on $\Omega \times (0,T)$ and $u = - \log f$. Then $u \in \PBMO$ with norm depending only on $C_0$, $C_1$, $p$ and $n$.
\end{lemma}
\begin{proof}
Let $R$ be a parabolic rectangle such that $ \sigma R \subset \Omega \times (0,T)$. Set
\[b = \min\{(p-1)/2,1\}.\] 
Then a straightforward integration gives
\begin{align*}
& \int_{  R^{+}} (u + \beta)_{+}^{b} \dmu = b\int_{0}^{\infty} \lambda^{b-1} \abs{\{(x,t) \in  R^{+} : u + \beta > \lambda \}} \dla \\
&= b\int_{0}^{\infty} \lambda^{b-1} \abs{\{(x,t) \in  R^{+} : -\log f + \beta  > (\lambda- C') + C' \}} \dla \\
&\leq  \abs{ R^{+}} (1+C')^{b} \\
&\hspace{1cm} + b\int_{1}^{\infty} (\lambda + C')^{b-1} \abs{\{(x,t) \in  R^{+} : -\log f + \beta  > \lambda + C' \}} \dla \\ 
&\leq   \abs{R^{+}} (1+C')^{b} + C b \abs{ R^{+}} \int_{1}^{\infty} (\lambda + C')^{b - p}  \dla  
\end{align*}
and 
\begin{align*}
&  \int_{ R^{-}}(u + \beta)_{-}^{b} \dmu = b\int_{0}^{\infty} \lambda^{b-1} \abs{\{(x,t) \in  R^{-} : -u  - \beta > \lambda \}} \dla \\
&= b\int_{0}^{\infty} \lambda^{b-1} \abs{\{(x,t) \in  R^{-} :  \log f - \beta  > (\lambda- C') + C' \}} \dla \\
&\leq   \abs{ R^{-}} (1+C')^{b} \\
&\hspace{1cm} + b\int_{1}^{\infty} (\lambda + C')^{b-1} \abs{\{(x,t) \in  R^{-} : \log f - \beta  > \lambda + C' \}} \dla \\ 
&\leq   \abs{ R^{-}} (1+C')^{b} + Cb \abs{ R^{-}} \int_{1}^{\infty} (\lambda + C')^{b-p}  \dla  
\end{align*}
so $u$ satisfies Definition \ref{def: lag mappings}:
\[\sup_{\sigma R \subset \Omega \times (0,T)} \left( \dashint_{  R^{+}} (u-a_R)_{+}^{b} \dmu +  \dashint_{  R^{-}} (a_R- u)_{+}^{b} \dmu \right) < \infty. \]
According to Remark \ref{remark: John-Nirenberg}, we get a John-Nirenberg lemma. Even if this differs from \eqref{JN+} and \eqref{JN-}, all the arguments of the previous sections are still valid up to change of some dimensional constants. Thus we may apply Theorem \ref{thm: reimann_rychener} to conclude that $u \in \PBMO( \Omega \times (0,T))$.
\end{proof}

Having established the fact $ u \in \PBMO(\Omega\times (0,T))$, global integrability of positive supersolutions follows easily.

\begin{theorem}
Let $f> \gamma > 0$ be a supersolution to \eqref{eq: equation} on $\Omega \times (0,T)$ where $\Omega \subset \mathbb{R}^{n}$ is a domain satisfying a quasihyperbolic boundary condition. Then for each $\delta > 0$ there is $\epsilon > 0$ depending only on $p$, $n$, $\delta$, $\Omega$, $C_0$ and $C_1$ such that
\begin{equation}
\label{eq: global integrability}
\int_{\Omega \times (0,T- \delta)} f ^{\epsilon} \dmu < \infty. 
\end{equation}
\end{theorem}
\begin{proof}
By Lemma \ref{lemma: logbmo} $-\log f \in \PBMO(\Omega\times (0,T))$, so by Corollary \ref{cor: ekspint} there are $c \in \mathbb{R}$ and $\epsilon > 0$ such that
\begin{align*}
\infty > \int_{\Omega \times (0,T- \delta)} e^{\epsilon(- \log f -c )^{-} } \dmu &\geq  \int_{\Omega \times (0,T- \delta)} e^{\epsilon((- \log f)^{-} - (c )^{-}) } \dmu \\
& = C \int _{\Omega \times (0,T- \delta) \cap \{ f > 1 \}} f^{\epsilon} \dmu  + C ,
\end{align*}
so the finiteness of the integral in \eqref{eq: global integrability} follows.
\end{proof}

The assumption $f > \gamma > 0$ coming from Lemma \ref{lemma:KinnunenKuusi} can be replaced by the assumption that $f > 0$ is lower semicontinuous. Indeed, for a lower semicontinuous $f > 0$ it actually holds that $f > \gamma_R > 0$ in all parabolic rectangles $R$. Since \ref{lemma:KinnunenKuusi} provides an estimate uniform in $\gamma$, we can actually apply it, and get that $- \log f \in \PBMO^{\sigma}$. Thus the previous theorem can also be stated in the following form.

\begin{theorem}
Let $f>   0$ be a lower semicontinuous  supersolution to \eqref{eq: equation} on $\Omega \times (0,T)$ where $\Omega \subset \mathbb{R}^{n}$ is a domain satisfying a quasihyperbolic boundary condition. Then for each $\delta > 0$ there is $\epsilon > 0$ depending only on $p$, $n$, $\delta$, $\Omega$, $C_0$ and $C_1$ such that
\begin{equation*}
\int_{\Omega \times (0,T- \delta)} f ^{\epsilon} \dmu < \infty. 
\end{equation*}
\end{theorem}

Another extension is to consider increasing limits of positive supersolutions. In this case it suffices to note that, in addition to finiteness, the integral in \eqref{eq: global integrability} has an upper bound uniform in the supersolutions $f$ except for the quantity $e^{c}$. Indeed, in addition to $e^{c}$ the only dependence is on $\norm{-\log f}_{\PBMO}$, which is determined by the structural constants $C_0$, $C_1$, $p$ and $n$. Thus the global integrability of increasing limits of positive supersolutions follows from the monotone convergence theorem provided that the functions in the sequence are uniformly bounded in $\Omega \times (T- \delta , T)$.

\nocite{AE1996}

\bibliographystyle{amsra}
\bibliography{viitteet}{}

\end{document}